\documentclass[12pt]{article}
\usepackage[english]{babel}
\usepackage{amsmath,amsthm,amsfonts,amssymb, array, epsfig, titlesec, hyperref, epsfig, float, caption}
\usepackage[left=0.7in,top=0.5in,right=0.7in]{geometry}
\usepackage[mathscr]{eucal}
\usepackage{graphicx}
\usepackage{subcaption}
\usepackage{hyperref}

\newtheorem{thm}{Theorem}[section]
\newtheorem{lem}[thm]{Lemma}

\newtheorem{rem}[thm]{Remark}

\theoremstyle{definition}

\newtheorem{thml}{Theorem}

\titleformat*{\section}{\normalsize \bfseries \filcenter}
\newcommand{\eps}{\varepsilon}

\newtheorem*{mainthm}{Main Theorem}
\newtheorem*{thm*}{Theorem}

\newcommand{\Addresses}{{
  \bigskip
  \footnotesize

  \textsc{Department of Mathematics, The Pennsylvania State University,
    University Park, PA 16802}\par\nopagebreak
  \textit{E-mail address}: \texttt{axe930@psu.edu}
}}

\begin{document}

\title{\normalsize \textbf{Flexibility of exponents for expanding maps on a circle}}
\author{\normalsize Alena Erchenko \thanks{This work was partially supported by NSF grant DMS 16-02409.}}
\date{}
\maketitle
\thispagestyle{empty}

\begin{abstract}
We consider a smooth expanding map $g$ on the circle of degree $2$. It is known that the Lyapunov exponent of $g$ with respect to the unique invariant measure that is absolutely continuous with respect to the Lebesgue measure is positive and less than or equal to $\log 2$ which, in addition, is less than or equal to the Lyapunov exponent of $g$ with respect to the measure of maximal entropy. Moreover, the equalities only occur simultaneously. We show that these are the only restrictions on the Lyapunov exponents considered above for smooth expanding maps of degree $2$. 
\end{abstract}

\section{Introduction}

Let $g: S^1\rightarrow S^1$ be a smooth expanding map of the circle $S^1 = \mathbb R / \mathbb Z$. Let $\mathcal M(g)$ denote the set of $g$-invariant Borel probability measures. For any $\mu\in\mathcal M(g)$ that is ergodic, its Lyapunov exponent is defined by $\lambda_{\mu}(g) = \int_{S_1}\log|g'|d\mu$. We concentrate our attention on the Lyapunov exponent $\lambda_{abs}(g)$ with respect to the unique measure in $\mathcal M (g)$ that is absolutely continuous with respect to the Lebesgue measure and the Lyapunov exponent $\lambda_{max}(g)$ with respect to the measure of maximal entropy.  Also, let $h_\mu(g)$ be the metric entropy of $g$ with respect to $\mu$ for any $\mu \in \mathcal{M}(g)$. 

The Ruelle inequality \cite{R78} shows that for any $\mu\in\mathcal M(g)$, we have $h_{\mu}(g)\leqslant \lambda_{\mu}(g)$. If $m\in\mathcal M(g)$ is the measure absolutely continuous with respect to the Lebesgue measure, then we have that $h_{m}(g)=\lambda_m(g)$ \cite{BP07}. The topological entropy of $g$ is equal to $h_{top}(g)=\sup\limits_{\mu\in\mathcal M(g)}h_{\mu}(g)$. In our setting, the supremum is achieved by a unique measure called the measure of maximal entropy. For the $\times 2$-map on $S^1$ given by $x\mapsto 2x \pmod 1$, we have that $\lambda_{abs}(\times 2)=h_{top}(\times 2)=\lambda_{max}(\times 2)=\log 2$ and the measure of maximal entropy is the Lebesgue measure. 

Recall that two continuous maps $f: S^1\rightarrow S^1$ and $g: S^1\rightarrow S^1$ are said to be topologically conjugate if there exists a homeomorphism $h: S^1\rightarrow S^1$ such that $f=h^{-1}\circ g\circ h$. A continuous expanding map on $S^1$ has degree $2$ if and only if it is topologically conjugate to $\times 2$ \cite{HK95}. Also, topological entropy is invariant under topological conjugacy. Therefore, it follows from all of the above that for any smooth expanding map $g:S^1 \to S^1$ of degree $2$, we have 
\[ \lambda_{abs}(g)\leqslant \log 2\leqslant \lambda_{max}(g). \]
If we also have that $\lambda_{abs}(g)=\log 2$, then we obtain that $h_{\mu}(g)=h_{top}(g)$, where $\mu$ is the measure which is absolutely continuous with respect to Lebesgue measure. Therefore, $\mu$ is the measure of maximal entropy for $g$ and $\lambda_{max}(g)=\lambda_{abs}(g)=\log 2$. On the other hand, assume $g$ is a smooth expanding map of degree $2$ and $\lambda_{max}(g)=\log 2$, i.e., $h_{top}(g)=\int_{S_1}\log|g'|d\mu$, where $\mu$ is the measure of maximal entropy. Then, by Theorem 3 in \cite{L81} we have that $\mu$ is the measure which is absolutely continuous with respect to Lebesgue measure implying that $\lambda_{abs}(g)=\lambda_{max}(g)=\log 2$. Thus, we have seen that the equalities in $\lambda_{abs}(g)\leqslant \log 2\leqslant \lambda_{max}(g)$ only hold simultaneously. It is a natural question if these inequalities are the only restrictions on the pair of values of the considered Lyapunov exponents. In the following theorem, we answer this question affirmatively.   

\begin{mainthm}\label{main_thm}
For any positive numbers $a, b$ such that $a < \log 2 < b$, there exists a smooth expanding map $g$ of degree $2$ such that $\lambda_{abs}(g)=a$ and $\lambda_{max}(g)=b$. 
\end{mainthm}

\begin{rem}
The proof of this theorem given below can be easily adapted to show that any pair of positive numbers separated by $\log n$ can be realized as Lyapunov exponents with respect to the measure absolutely continuous with respect to the Lebesgue measure and the measure of maximal entropy of a smooth expanding map of degree $n$ for any natural number $n>1$.
\end{rem}

The main theorem is another example in the flexibility program proposed by A. Katok. The program studies the expectation that classical smooth systems (diffeomorphisms and flows) are quite flexible in comparison to actions of higher rank abelian groups. The first example in this direction was obtained in the work by the author and A. Katok \cite{EK} which shows the flexibility for the values of the pair of metric entropy with respect to Liouville measure and topological entropy for geodesic flow on surfaces of negative curvature with fixed genus greater than or equal to 2 and fixed total area. Another result in the flexibility program is due to J.Bochi, F. Rodriguez Hertz and A.Katok \cite{BKRH} who show how to vary the Lyapunov exponents with respect to the Lebesgue measure for volume-preserving Anosov diffeomorphisms with dominated splittings into one-dimensional bundles.

There are still many open questions related to the flexibility program and many properties of smooth dynamical systems whose flexibility is unknown. In particular, a natural extension of the main theorem would be to consider a similar problem on the torus $\mathbb T^2 = \mathbb R^2 / \mathbb Z^2$. Let $L_A$ be an Anosov linear area-preserving automorphism of $\mathbb T^2$. For $L_A$, we have that the Lyapunov exponent $\lambda_{Leb}(L_A)$ with respect to the Lebesgue measure, topological entropy $h_{top}(L_A)$, and the Lyapunov exponent with respect to the measure of maximal entropy $\lambda_{max}(L_A)$ coincide. Let us consider a smooth Anosov area-preserving diffeomorphism $g$ on $\mathbb T^2$ homotopic to $L_A$. Then, we have that 
\[ \lambda_{Leb}(g)\leqslant h_{top}(g)=h_{top}(L_A)\leqslant \lambda_{max}(g). \]
The question is if these inequalities are the only restriction on the considered Lyapunov exponents for smooth Anosov area-preserving diffeomorphisms on $\mathbb T^2$ homotopic to $L_A$. We hope to answer this question in future work. 

The rest of the paper consists of the proof of the main theorem. First, we construct continuous piecewise linear maps of degree $2$ which realize all possible values for the pairs of the considered Lyapunov exponents in Theorem~\ref{thm_piecewise_linear}. We then give a procedure for smoothing these maps to complete the proof.

\vspace{.5cm}

\textbf{Acknowledgments.} The author would like to thank Anatole Katok for introducing her to the flexibility program and its many interesting problems. She also thanks Federico Rodriguez Hertz for many helpful discussions and comments regarding the constructions presented here.

\section{Proof of the Main theorem}\label{proof_thm}

Let $g$ be a continuous degree $2$ piecewise linear expanding map on $S^1$. Note that $g$ has a unique measure of maximal entropy. We will say that $g$ is well-behaved if it has a unique invariant measure absolutely continuous with respect to the Lebesgue measure. All of the inequalities discussed in the introduction will still hold for a well-behaved $g$. We will begin the proof by constructing well-behaved continuous piecewise linear expanding maps of $S^1$ that take on all possible values of pairs of Lyapunov exponents for the absolutely continuous with respect to Lebesgue and maximal entropy measures. 

\begin{thm}\label{thm_piecewise_linear}
For any positive numbers $a, b$ such that $a < \log 2 < b$, there exists a well-behaved continuous piecewise linear expanding map $g$ of degree $2$ such that $\lambda_{abs}(g)=a$ and $\lambda_{max}(g)=b$. 
\end{thm}

\begin{proof}
We define the following family of functions $f(x; n, \delta; k, \eps)$ on a normalized circle $S^1 = \mathbb R / \mathbb Z$, where $n$ and $k$ are integers greater than or equal to $2$, $\delta\in[2^{-n-1};2^{-n})$ and $\eps\in[2^{-k-1};2^{-k})$. The graph of the function $f(x; n, \delta; k, \eps)$ is displayed in Figure~\ref{new_map}.

\[
f(x; n, \delta; k, \eps) = \left\{
  \begin{aligned}
    &  \frac{1}{2^n\delta}x\quad& \text{ if } \quad & x \in [0;\delta),\\
    &  \frac{1}{1-2^n\delta}x+\left(2^{-n}-\frac{\delta}{1-2^n\delta}\right)\quad& \text{ if } \quad & x \in [\delta;2^{-n}),\\
    &  \frac{1}{1-2^k\eps}x+(2^{-1}-2^{-k})\left(2-\frac{1}{1-2^k\eps}\right)\quad& \text{ if } \quad & x \in [2^{-1}-2^{-k};2^{-1}-\eps),\\
    &  \frac{1}{2^k\eps}x+\left(1-\frac{1}{2^{k+1}\eps}\right)\quad& \text{ if } \quad & x \in [2^{-1}-\eps;2^{-1}),\\
		& 2x \pmod 1 \qquad& &\text{ otherwise }.\\
  \end{aligned} \right.
\]

Observe that $f(x; n, 2^{-n-1}; k, 2^{-k-1}) = 2x \pmod 1$ for any values of parameters $n$ and $k$ that we have allowed. Moreover, the maps $f(x; n, \delta; k, \eps)$ are of degree $2$. The maps $f(x; n , \delta; k, \eps)$ are well-behaved by \cite{BS79} as the partition of $S^1$ into pieces where $f$ is linear can be refined to a Markov partition.

\begin{figure}[H]
  \centering
  \includegraphics[width=.5\linewidth]{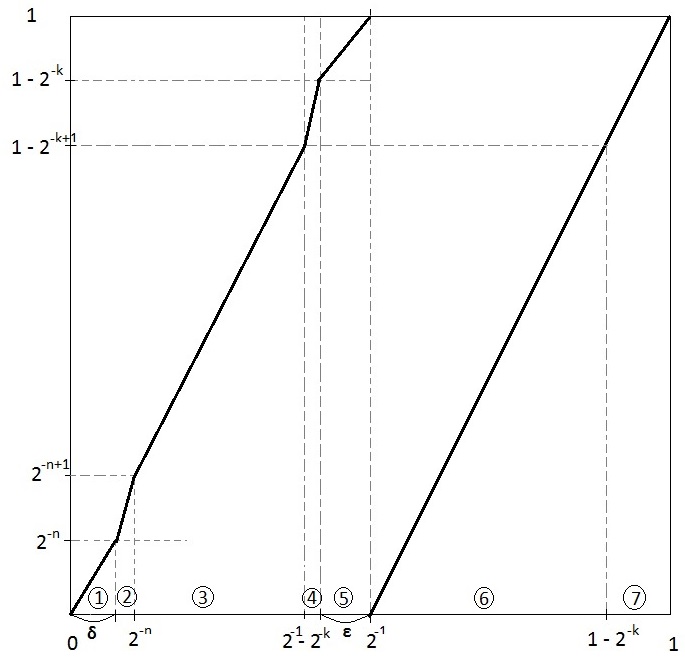}
    \caption{A representative of the constructed family of maps of degree $2$}
    \label{new_map}
  \end{figure}

Let us denote by $I_j$ the segment labeled in Figure~\ref{new_map} by the circled number $j$, where $j=1, 2, \hdots, 7$. The preimages of the fixed point $x=0$ generate a Markov partition for $f(x; n, \delta; k, \eps)$. By the choice of functions $f(x; n, \delta; k, \eps)$, we have that the non-smooth points of $f(x; n, \delta; k, \eps)$ are exactly at endpoints of intervals of this Markov partition. Therefore, we obtain for the measure of maximal entropy (Parry measure \cite{HK95}) that the measures of $I_1$ and $I_2$ are equal to $2^{-(n+1)}$, the measures of $I_4$ and $I_5$ are equal to $2^{-(k+1)}$, and the total measure of the rest is $1-2^{-n}-2^{-k}$. Thus, we can compute the Lyapunov exponent of $f(x; n, \delta; k, \eps)$ with respect to the measure of maximal entropy.

\begin{align}\label{max_entropy_exp}
\lambda_{max}(f(x; n, \delta; k, \eps)) &= 2^{-(n+1)}\log\left(\frac{1}{2^n\delta}\right)+2^{-(n+1)}\log\left(\frac{1}{1-2^n\delta}\right)+ \nonumber\\&+2^{-(k+1)}\log\left(\frac{1}{2^k\eps}\right)+2^{-(k+1)}\log\left(\frac{1}{1-2^k\eps}\right)+(1-2^{-n}-2^{-k})\log 2 = \\
&= -\frac{\log(2^n\delta)+\log(1-2^n\delta)}{2^{n+1}}-\frac{\log(2^k\eps)+\log(1-2^k\eps)}{2^{k+1}}+(1-2^{-n}-2^{-k})\log 2 \nonumber
\end{align}

To be able to compute the Lyapunov exponent of $f(x; n, \delta; k, \eps)$ with respect to the probability invariant measure $\mu$ which is absolutely continuous with respect to the Lebesgue measure, we need to find this invariant measure, i.e., the invariant density $q$. For simplicity, we omit the dependence of the measure $\mu$ and the function $q$ on the parameters $n, \delta, k$ and $\eps$ in our notation. For each $f(x; n, \delta; k, \eps)$, we try to find $q(x)$ in a special form such as $q(x) = a_j$ if $x\in I_j$, where $a_j$ is some constant and $j=1, 2, \hdots, 7$. This choice of special form is natural as by Theorem 3 in \cite{BS79} we know that for $f(x; n, \delta; k, \eps)$ the unique absolutely continuous invariant density is piecewise constant. We find $a_j$, where $j=1,\dots,7$, from the conditions that the measure is invariant, i.e., the measure of the set is equal to the measure of its preimage under $f(x; n, \delta; k, \eps)$.

We obtain the $a_j$ by examining the conditions needed for the measure of sets of the form $[0; x]$ to be preserved. Note that in order to show that an absolutely continuous measure with respect to the Lebesgue measure is invariant, it is enough to check these sets. We have the following cases.
\begin{enumerate}
\item If $x\in I_1$, we need $a_1 x = a_1 2^n\delta x + a_6\frac{x}{2}$. Therefore, we obtain that $a_6 = 2a_1(1-2^n\delta).$
\item If $x\in I_2$, we need $a_1\delta + a_2(x-\delta) = a_1 2^n\delta x +a_6\frac{x}{2}$. As a result, we obtain that $a_2 = a_1.$
\item If $x\in [2^{-n}; 2^{-n+1})$, we need $$a_1\delta+a_2(2^{-n}-\delta)+a_3(x-2^{-n}) = a_1\delta+a_2(x-2^{-n})(1-2^n\delta)+a_6\frac{x}{2}.$$ By the previous items, this equality is equivalent to the following equalities.
$$a_1(2^{-n}-\delta)+a_3(x-2^{-n}) = a_1(x-2^{-n})(1-2^n\delta)+a_1(1-2^n\delta)x$$
$$2^{-n}a_1(1-2^n\delta)+a_3(x-2^{-n}) = 2a_1(x-2^{-n})(1-2^n\delta)+a_1 2^{-n}(1-2^n\delta)$$
Therefore, we obtain that $a_3 = 2a_1(1-2^n\delta)=a_6$. 
\item If $x\in [2^{-n+1}; 2^{-1}-2^{-k+1})$, we need $$a_1\delta+a_2(2^{-n}-\delta)+a_3(x-2^{-n}) = a_1\delta +a_2(2^{-n}-\delta)+a_3\frac{x-2^{-n+1}}{2}+a_6\frac{x}{2}.$$ This equality is guaranteed by the conditions on $a_1, a_2, a_3, a_6$ from the previous items.
\item If $x\in I_4$, we need 
$$a_3(2^{-1}-2^{-k}-2^{-n})+a_4(x-(2^{-1}-2^{-k})) = a_3\frac{x-2^{-n+1}}{2}+a_6\frac{x}{2}.$$
By the previous items, this equality is equivalent to the following equality.
$$a_3(2^{-1}-2^{-k}-2^{-n})+a_4(x-(2^{-1}-2^{-k})) = a_3(x-(2^{-1}-2^{-k}))+a_3(2^{-1}-2^{-k}-2^{-n}).$$
As a result, we obtain $a_4 = a_3 = 2a_1(1-2^n\delta)$.
\item If $x\in I_5$, we need
$$a_3(2^{-1}-2^{-k}-2^{-n})+a_4(2^{-k}-\eps)+a_5(x-(2^{-1}-\eps)) = a_3\frac{x-2^{-n+1}}{2}+a_6\frac{x}{2}.$$
By the previous items, this equality is equivalent to the following equality.
$$a_3(2^{-1}-2^{-n}-\eps)+a_5(x-(2^{-1}-\eps)) = a_3(x-(2^{-1}-\eps))+a_3(2^{-1}-\eps-2^{-n}).$$
Therefore, $a_5 = a_3 = 2a_1(1-2^n\delta)$.
\item If $x\in [2^{-1}; 1-2^{-k+1})$, we need
$$a_3(2^{-1}-2^{-k}-2^{-n})+a_4(2^{-k}-\eps)+a_5\eps+a_6(x-2^{-1}) = a_3\frac{x-2^{-n+1}}{2}+a_6\frac{x}{2}.$$
This equality holds as we already have that $a_3 = a_4 = a_5 = a_6$.
\item If $x\in [1-2^{-k+1}, 1-2^{-k})$, we need
$$a_4(2^{-k}-\eps)+a_5\eps+a_6(x-2^{-1}) = a_4(x-(1-2^{-k+1}))(1-2^k\eps)+a_6(2^{-1}-2^{-k})+a_7\frac{x-(1-2^{-k+1})}{2}.$$
By the previous items, this equality is equivalent to the following equality.
$$a_3 2^k\eps(x-(1-2^{-k+1})) = a_7\frac{x-(1-2^{-k+1})}{2}.$$
As a result, we obtain $a_7 = 2a_3 2^k\eps = 4a_1(1-2^n\delta)2^k\eps$.
\item If $x\in I_7$, we need
$$a_5\eps+ a_6(2^{-1}-2^{-k})+a_7(x-(1-2^{-k})) = a_5(x-(1-2^{-k}))2^k\eps + a_6(2^{-1}-2^{-k})+a_7\frac{x-(1-2^{-k+1})}{2}.$$
This equality holds under the previous conditions on the constants.
\end{enumerate}

The total measure of the circle $S_1$ for the measure $\mu$ is equal to $1$ if 
$$a_1 = \frac{2^n}{1+2(1-2^n\delta)(2^n-2^{n-k}-1+2^{n-k+1}\cdot 2^k\eps)}.$$

From this calculation, we have that the Lyapunov exponent of $f(x; n, \delta; k, \eps)$ with respect to the probability invariant measure that is absolutely continuous with respect to Lebesgue measure is exactly the following.

\begin{align}\label{abs_exp}
&\lambda_{abs}(f(x; n, \delta; k, \eps)) = \nonumber \\
& = \frac{1}{1+2(1-2^n\delta)(2^n-2^{n-k}-1+2^{n-k+1}\cdot 2^k\eps)}(2(1-2^n\delta)(2^n-2^{n-k+1}-1+2^{n-k+1}\cdot 2^k\eps)\log 2 - \\
&-(2^n\delta \log(2^n\delta)+(1-2^n\delta)\log(1-2^n\delta)) - 2^{n-k+1}(1-2^n\delta)(2^k\eps\log(2^k\eps)+(1-2^k\eps)\log(1-2^k\eps))) \nonumber
\end{align}

Notice that the bounds for $\delta$ depend on $n$ and the bounds for $\eps$ depend on $k$. If we define $u = 2^n\delta$ and $v = 2^k\eps$, then $u$ and $v$ are independent of $n$ and $k$, respectively, and vary in the interval $[\frac{1}{2};1)$. We then consider functions $g(x; n, u, k, v):=f(x; n, 2^{-n}u; l, 2^{-k}v)$. Therefore, we obtain a 4-parameter family of functions $g(x; n, u, k, v)$ which is continuous in $u$ and $v$ for fixed $n$ and $k$. Notice that $g(x; n, \frac{1}{2}, k, \frac{1}{2})$ is the $\times 2$-map for every $n$ and $k$. 

We now show that for any positive numbers $a, b$ such that $a < \log 2 < b$, there exist integers $\bar n, \bar k$ which are greater than 2 and numbers $\bar u, \bar v$ in the interval $[\frac{1}{2};1)$ such that $\lambda_{abs}(g(x;\bar n, \bar u, \bar k, \bar v))=a$ and $\lambda_{max}(g(x;\bar n, \bar u, \bar k, \bar v))=b$.

Rewriting expressions that we obtained for the exponents of $g$ in terms of $u$ and $v$ gives the following.

\begin{equation}\label{max_exp_g}
\lambda_{max}(g(x; n, u, k, v))= -\frac{\log(u)+\log(1-u)}{2^{n+1}}-\frac{\log(v)+\log(1-v)}{2^{k+1}}+(1-2^{-n}-2^{-k})\log 2
\end{equation}

\begin{align}\label{abs_exp_g}
&\lambda_{abs}(g(x; n, u, k, v)) = \nonumber \\
& = \frac{1}{1+2(1-u)(2^n-2^{n-k}-1+2^{n-k+1}v)}(2(1-u)(2^n-2^{n-k+1}-1+2^{n-k+1}\cdot v)\log 2 - \\
&-(u \log u+(1-u)\log(1-u)) - 2^{n-k+1}(1-u)(v\log v+(1-v)\log(1-v))) \nonumber
\end{align}

To show that all possible values of pairs of the considered Lyapunov exponents are realizable, we analyze the boundary cases of parameters. Consider the maps $g(x; n, \frac{1}{2}, k, v)$. These maps are independent of $n$ and act as the $\times 2$-map everywhere except in the interval $[2^{-1}-2^{-k}; 2^{-1})$. First, we will examine the Lyapunov exponent $\lambda_{abs}(g)$ which has the following expression.

\begin{align}
&\lambda_{abs}(g(x; n, \frac{1}{2}, k, v)) = \nonumber \\
&=\frac{1}{1-2^{-k}+2^{-k+1}v}((1-2^{-k+1}+2^{-k+1}v)\log 2 - 2^{-k}(v\log v+(1-v)\log(1-v)))
\end{align}

Notice that the values of $-v \log v - (1-v)\log(1-v)$ belong to the interval $(0; \log 2]$ when $v\in[\frac{1}{2};1)$. The maximum value of that function is attained at $v=\frac{1}{2}$. Also, $-v \log v - (1-v)\log(1-v)$ tends to $0$ as $v$ tends to $1$.

We have that
\begin{equation}
\log 2\geqslant \lambda_{abs}(g(x; n, \frac{1}{2}, k, v))\geqslant \left(1-\frac{2^{-k}}{1-2^{-k}+2^{-k+1}v}\right)\log 2\geqslant \left(1-\frac{1}{2^k}\right)\log 2 
\end{equation}
for every $n$, $k$ and $v\in[\frac{1}{2};1)$. The lower bound on $\lambda_{abs}(g(x; n, \frac{1}{2}, k, v))$ tends to $\log 2$ as $k$ tends to $+\infty$.

Now, let us consider the Lyapunov exponent with respect to the measure of maximal entropy.

\begin{equation}
\lambda_{max}(g(x; n, \frac{1}{2}, k, v)) = (1-2^{-k})\log 2 + \frac{1}{2^{k+1}}(-\log v -\log(1-v))
\end{equation}

 For any fixed $k$, $\lambda_{max}(g(x; n, \frac{1}{2}, k, v))$ is monotonically increasing for $v\in[\frac{1}{2};1)$ and its values vary from $\log 2$ to $+\infty$ as $v$ varies from $\frac{1}{2}$ to $1$.  As a result, by choosing $k$ sufficiently large, we obtain a family of maps $g(x; n, \frac{1}{2}, k, v)$ parametrized by $v\in[\frac{1}{2};1)$
such that $\lambda_{abs}(g)$ is close to $\log 2$ and for any $b\geqslant \log 2$ there exists $\bar v\in[\frac{1}{2};1)$ such that $\lambda_{max}(g(x; n, \frac{1}{2}, k, \bar v)) = b$.

Next, consider the maps $\bar g (x; n, u):= g(x; n, u, k, \frac{1}{2})$. These maps are independent of $k$ and act as the $\times 2$-map everywhere except in the interval $[0; 2^{-n})$. We will examine the Lyapunov exponents as before. We first have the following expression for $\lambda_{abs}$.

\begin{equation}
\lambda_{abs}(\bar g(x; n, u)) = \log 2 + \frac{1}{1+2(1-u)(2^n-1)}\left(-u\log u - (1-u)\log(1-u)-\log 2\right)
\end{equation}

As we noticed before, the values of $-u \log u - (1-u)\log(1-u)$ belong to the interval $(0; \log 2]$ when $u\in[\frac{1}{2};1)$. For fixed $n$, we have that $\lambda_{abs}(\bar g(x; n, \frac{1}{2})) = \log 2$ and $\lambda_{abs}(\bar g(x; n, u))$ tends to $0$ as $u$ tends to $1$.

We now show that $\lambda_{abs}(\bar g(x; n, u))$ monotonically decreases when $u$ varies from $\frac{1}{2}$ to $1$. The derivative of $\lambda_{abs}(\bar g(x; n, u))$ with respect to $u$ is the following.

\begin{align}
\frac{\partial}{\partial u}\lambda_{abs}(\bar g(x; n, u)) &= \frac{2(2^n-1)}{(1+2(1-u)(2^n-1))^2}(-u\log u-(1-u)\log(1-u)-\log 2)+\nonumber\\ &+\frac{1}{1+2(1-u)(2^n-1)}(-\log u -1 +\log(1-u)+1) = \\ &=\frac{1}{(1+2(1-u)(2^n-1))^2}\left(-2(2^n-1)\log2+\log(1-u)-(2^{n+1}-1)\log u\right) \nonumber
\end{align}

Let us denote $y(u) = -2(2^n-1)\log2+\log(1-u)-(2^{n+1}-1)\log u$. Then, $y(\frac{1}{2}) = 0$. Also, we have that its derivative is equal to
$$y'(u) = -\frac{1}{1-u}-\frac{2^{n+1}-1}{u} = -\frac{-2(2^n-1)u+(2^{n+1}-1)}{u(1-u)}<-\frac{1}{u(1-u)}<0, $$
as $u\in[\frac{1}{2};1)$. It follows, that $y(u)<0$ when $u\in[\frac{1}{2};1)$.

Therefore, we obtain that $\frac{\partial}{\partial u}\lambda_{abs}(\bar g(x; n, u))<0$, i.e., $\lambda_{abs}(\bar g(x; n, u))$ is a monotonically decreasing function with respect to $u$.

Now, let us examine the Lyapunov exponent of $\bar g(x; n, u)$ with respect to the measure of maximal entropy.

\begin{equation} \label{quad}
\lambda_{max}(\bar g(x; n, u)) = (1-2^{-n})\log 2 + \frac{1}{2^{n+1}}(-\log u -\log(1-u))
\end{equation}

For any fixed $n$, $\lambda_{max}(\bar g(x; n, u))$ is monotonically increasing for $u\in[\frac{1}{2};1)$ and its values vary from $\log 2$ to $+\infty$ as $u$ varies from $\frac{1}{2}$ to $1$.

We claim that for every $\alpha, \beta >0$ there exists an integer $N\geqslant 2$ and $u_{N} \in[\frac{1}{2}; 1)$ such that $\lambda_{max}(\bar g(x; N, u_N))-\log 2\leqslant\alpha$ and $\lambda_{abs}(\bar g(x; N, u_N))<\beta$.  To see this, we want to express $u$ in terms of $\lambda_{max}(\bar g(x; n, u))$. First, we obtain the quadratic equation
$$u^2 - u + e^{2(2^n-1)\log 2 - 2^{n+1}\lambda_{max}(\bar g(x; n, u))}=0 $$
from equation \ref{quad}. Then, as $u\in[\frac{1}{2};1)$, we have
$$u = \frac{1+\sqrt{1-e^{-2^{n+1}(\lambda_{max}(\bar g(x; n, u)) - \log 2)}}}{2}.$$

If $\lambda_{max}(\bar g(x; N, u_N)) - \log 2 = \alpha$, then $u_N = \frac{1+\sqrt{1-e^{-2^{N+1}\alpha}}}{2}$. Note that $u_N\sim 1-\frac{1}{4}e^{-2^{N+1}\alpha}\rightarrow 1$ as $N\rightarrow +\infty$. Moreover,

\begin{equation}
\lambda_{abs}(\bar g(x; N, u_N)) = \log 2 + \frac{1}{1+2(1-u_N)(2^N-1)}\left(-u_N\log u_N - (1-u_N)\log(1-u_N)-\log 2\right).
\end{equation}

We observe that $-u_N\log u_N - (1-u_N)\log(1-u_N)$ tends to $0$ as $u_N$ tends to $1$. Also,
$$2(1-u_N)(2^N-1)\sim \frac{1}{2}e^{-2^{N+1}\alpha}(2^N-1)\rightarrow 0 \text{ as } N\rightarrow +\infty,$$
because $\lim_{z\rightarrow +\infty}\frac{1}{2}e^{-2z\alpha}(z-1)=0$ if $\alpha\neq 0$. If $\lambda_{max}(\bar g(x; N, u_N)) - \log 2 = \alpha$ and $N$ is large enough, then $\lambda_{abs}(\bar g(x; N, u_N))<\beta$.

As a result, we have verified that for every $a, b$ such that $0<a<\log 2< b$ there exists an integer $N\geqslant 2$ such that $\lambda_{max}(\bar g(x; N, u_N)) - \log 2  = \frac{b-\log 2}{2}$ and $\lambda_{abs}(\bar g(x; N, u_N))<\frac{a}{2}$. From this result and the monotonicity properties of the functions $\lambda_{max}(\cdot)$ and $\lambda_{abs}(\cdot)$ with respect to $u$, we have that $\lambda_{max}(\bar g(x; N, u))-\log 2<\frac{b-\log 2}{2}$ for every $u\in[\frac{1}{2}; u_N]$ and for any $\alpha\in[\frac{a}{2};\log 2]$ there exists $\bar u\in[\frac{1}{2}; u_N]$ such that $\lambda_{abs}(\bar g(x; N, \bar u))=\alpha$.    

Also, there exists an integer $K\geqslant 2$ such that $\lambda_{abs}(g(x; n, \frac{1}{2}, K, v))>a$ for every $v\in [\frac{1}{2}; 1)$. In addition, observe that $\lambda_{max}(g(x;n,u, k, v))$ is monotonic increasing in $u \in [1/2; 1)$ for fixed $n, k$ and $v$. Using this observation, the monotonicity properties of $\lambda_{abs}(g(x; n, u, k, \frac{1}{2}))$ and $\lambda_{max}(g(x; n, u, k, \frac{1}{2}))$ with respect to $u$, and the continuity of $\lambda_{abs}(g(x; n, u, k, v))$and $\lambda_{max}(g(x; n, u, k, v))$ with respect to $u$ and $v$ for every $n$ and $k$, we complete the proof by deducing the following statement. There exists $\bar u,\bar v\in[\frac{1}{2};1)$ such that $\lambda_{abs}(g(x; N, \bar u, K, \bar v)) = a$ and $\lambda_{max}(g(x; N, \bar u, K, \bar v))=b$.
\end{proof}

Now, we are left to smooth the piecewise linear maps that we have constructed. Let us consider a continuous family of maps $\{f_{s,t}\}$, $s,t\in[0;1]$, on $S^1$ coming from Theorem~\ref{thm_piecewise_linear} such that these maps realize all pairs of values for the considered Lyapunov exponents in some neighborhood of the pair $(a,b)$, where the first coordinate is the value of $\lambda_{abs}$ and the second coordinate is the $\lambda_{max}$. There exists a smooth family (with respect to the parameters) of smooth expanding maps $\{f^{\alpha}_{s,t}\}$, $s,t\in[0;1]$, $\alpha\in (0;\alpha_{(a,b)}]$, where $\alpha_{(a,b)}$ is sufficiently small, such that $f^{\alpha}_{s,t}$ coincides with $f_{s,t}$ outside $\alpha$-neighborhoods of finitely many points (the points where $f_{s,t}$ is non-smooth) and $f^{\alpha}_{s,t}$ converges to $f_{s,t}$ uniformly as $\alpha$ tends to $0$. The existence follows from the fact that two linear maps can be smoothed at a point where their values agree with the derivative of the smoothed map in this neighborhood varying between the values of the derivative of the corresponding two linear pieces. 

In a smooth family of maps such as $f^{\alpha}_{s,t}$, the values of the Lyapunov exponents of interest vary continuously. As the family of maps is smooth, we have that the derivative of $f^{\alpha}_{s,t}$ on $S^1$, the fixed point and its preimages vary continuously in $s$, $t$ and $\alpha$. By the definition of the Lyapunov exponent and the construction of the measure of maximal entropy, we obtain that the Lyapunov exponent with respect to the measure of maximal entropy varies continuously. Moreover, smooth expanding maps on a circle have a unique absolutely continuous invariant measure and the invariant density is the fixed point of the Frobenius-Perron operator, which implies that it varies continuously (with respect to all the parameters) for $f^{\alpha}_{s,t}$. We would like to show that for sufficiently small $\alpha$, we can realize an open set of pairs of the Lyapunov exponents of interest such that the point $(a,b)$ lies inside that set. The main theorem then follows. 

From the construction of $f^{\alpha}_{s,t}$, we have that $\lambda_{max}(f^{\alpha}_{s,t})\rightarrow \lambda_{max}(f_{s,t})$ as $\alpha$ tends to $0$ for every $s,t\in[0;1]$ as any Markov partition (with elements of arbitrarily small length) for $f^{\alpha}_{s,t}$ converges to the corresponding Markov partition for $f_{s,t}$ as $\alpha\rightarrow 0$ (corresponding boundary points of the partitions converge). The Markov partitions we are considering are the partitions which we get for each map $f^\alpha_{s,t}$ and $f_{s,t}$ from the preimages of the fixed point. In particular, the fixed point of $f^\alpha_{s,t}$ converges to the fixed point of $f_{s,t}$.

Let us consider what happens with the Lyapunov exponent with respect to the probability invariant measure which is absolutely continuous with respect to Lebesgue measure. In \cite{GB89} P. G\'{o}ra and A. Boyarsky prove the following compactness theorem for the set of invariant densities for any family of expanding piecewise smooth maps.

\begin{thml}[\cite{GB89}, Theorem 1]
Let $\{\tau_{\alpha}\}_{\alpha\in \mathcal A}$ be a family of expanding piecewise smooth transformations satisfying the following conditions:
\begin{enumerate}
\item There is a constant $\lambda>1$ such that $|\tau'_\alpha(x)|\geqslant \lambda,$ whenever the derivative exists for any $\alpha\in \mathcal A$;
\item There exists a constant $W>0$ such that for any $\alpha\in\mathcal A$ we have the variation $Var\left|\frac{1}{\tau'_\alpha}\right|\leqslant W$;
\item There exists a constant $\delta>0$ such that for any $\alpha\in\mathcal A$, there exists a finite partition $\mathcal K_\alpha$ such that for $I\in\mathcal K_\alpha$, $\tau_\alpha|_I$ is one-to-one, $\tau_\alpha(I)$ is an interval, and $\min\limits_{I\in\mathcal K_\alpha} diam (I)>\delta$;
\item For any $m\geqslant 1$, there exists $\delta_m>0$ such that if $\mathcal K_\alpha^{(m)}=\bigvee_{j=0}^{m-1}\tau^{-j}_{\alpha}(\mathcal K_\alpha)$, then 

$\min\limits_{I\in\mathcal K^{(m)}_{\alpha}} diam(I)\geqslant\delta_m>0$.
\end{enumerate}
Then, any $\tau_\alpha, \alpha\in \mathcal A$, admits an invariant density $q_{\alpha}$, and the set $\{q_{\alpha}\}_{\alpha\in\mathcal A}$ is of uniformly bounded variation and hence is precompact in $L^1$. 
\end{thml}

For every $s,t\in[0;1]$, we can apply the theorem above for the family of maps $f^{\alpha}_{s,t}$, where $A = [0;\alpha_{(a,b)}]$ and $f^{0}_{s,t} = f_{s,t}$. All the conditions are obviously satisfied with $\mathcal K_\alpha$ being the Markov partition given by the preimages of the fixed points. As a result, we get the following.

\begin{lem}
For every $s,t\in[0;1]$ let $q^{\alpha}_{s,t}$ be the invariant densities for the expanding maps $f^{\alpha}_{s,t}$ (constructed above), where $\alpha\in[0;\alpha_{(a,b)}]$. Then, we have that $q^{\alpha}_{s,t}$ converges to $q^0_{s,t}$ in $L^1$ as $\alpha$ tends to $0$.
\end{lem}
\begin{proof}
We need to slightly adapt the proofs of Lemma 4 and Theorem 3 in \cite{GB89}, which imply that any limit point of $\{q^{\alpha}\}_{\alpha\in(0;\alpha_{(a,b)}]}$ is an invariant density of $f_{s,t}$. From the uniqueness of the measure absolutely continuous with respect to the Lebesgue measure for $f_{s,t}$, the lemma follows. The modification of the proof of Lemma 4 in \cite{GB89} consists of considering for each $f^{\alpha}_{s,t}$ the Markov partitions coming from the preimages of the fixed point which converge to the Markov partition of $f_{s,t}$.   
\end{proof}

Finally, let $B^{\alpha}_{s,t}$ be a set of pairs of the Lyapunov exponents of interest realized by $f^{\alpha}_{s,t}$, and $\partial B^{\alpha}_{s,t}$ be its boundary. For every $s,t\in[0;1]$, we obtain that $\lambda_{abs}(f^{\alpha}_{s,t})\rightarrow \lambda_{abs}(f_{s,t})$ and $\lambda_{max}(f^{\alpha}_{s,t})\rightarrow \lambda_{max}(f_{s,t})$ as $\alpha\rightarrow 0$. Note that $B^0_{s,t}$ is a square-like set and contains a point $(a,b)$ in the interior. Therefore, the winding number of $\partial B^0_{s,t}$ around $(a,b)$ is non-zero. Moreover, by the construction of $f^{\alpha}_{s,t}$ (smoothing of $f^0_{s,t}$ in $\alpha$-neighborhoods of finitely many points), $\partial B^{\alpha}_{s,t}$ converges to $\partial B^0_{s,t}$ uniformly as $\alpha\rightarrow 0$. Thus, there exists a sufficiently small $\beta$ such that the winding number of $\partial B^{\beta}_{s,t}$ around $(a,b)$ is non-zero. As a result, by the continuity of the Lyapunov exponents of interest, $B^{\beta}_{s,t}$ contains the point $(a,b)$.

\Addresses
\end{document}